\documentclass[preprint]{elsarticle}

\usepackage{ifpdf}
\ifpdf
  \usepackage{pdfsync}
\fi
\usepackage{etex}
\usepackage[latin1]{inputenc}
\usepackage[T1]{fontenc}
\usepackage[english]{babel}
\usepackage[fleqn]{amsmath}
\usepackage{amssymb, amscd, amsthm}
\usepackage{mathrsfs}
\usepackage{placeins}
\usepackage{tikz}
\usepackage{cases}
\usepackage{multicol}
\usepackage{enumerate}
\usepackage{enumitem}
\usepackage{rotating}
\usetikzlibrary{matrix}
\usepackage{booktabs}
\usepackage{array}
\usepackage{xcolor}
\usepackage[dvips, all,  color]{xy}
\usepackage{longtable}
\usepackage{ltcaption}
\usepackage{caption}
\usepackage{fancyhdr}
\usepackage{todonotes}
\usepackage{lscape}
\newcounter{enumglobal}

\newcommand{\fr}{\mathfrak}
\newcommand{\und}{\underline}
\newcommand{\ove}{\overline}

\def\re{\color{red}}
\def\blu{\color{cyan}}

\def\deg{\operatorname{deg}}

\def\down{\vee}
\def\up{\wedge}
\def\id{\operatorname{id}}
\def\Id{\operatorname{Id}}
\def\C{{\mathbb C}}
\def\F{{\mathbb F}}
\def\D{{\mathcal D}}
\def\A{{\mathcal A}}

\def\K{{\mathcal K}}

\def\Z{{\mathbb Z}}

\def\O{{\mathcal O}}

\def\hom{{\operatorname{Hom}}}
\def\Ext{{\operatorname{Ext}}}

\def\Mod{{\operatorname{Mod}}}
\def\mod{{\operatorname{mod}}}
\def\gmod{{\operatorname{gmod}}}

\def\d{{\operatorname{d}}}
\def\nes{{\operatorname{nes}}}
\def\eps{{\varepsilon}}
\def\phi{{\varphi}}
\def\emptyset{{\varnothing}}
\def\la{{\lambda}}
\def\La{{\Lambda}}

\def\ga{{\gamma}}

\def\al{{\alpha}}
\def\be{{\beta}}
\def\alch{{\al \check{ \ }}}
\def\p{{\mathfrak p}}
\numberwithin{equation}{section}
\newtheorem{satz}{Satz}[section]
\newtheorem{Theorem}[satz]{Theorem}

\newtheorem{Lemma}[satz]{Lemma}
\newtheorem{Remark}[satz]{Remark}

\newtheorem{Prop}[satz]{Proposition}
\newtheorem{Conjecture}[satz]{Conjecture}
\theoremstyle{definition}
\newtheorem{Example}[satz]{Example}
\newtheorem{Def}[satz]{Definition}

\newtheorem{Convention}[satz]{Convention}

\journal{Arxiv}

\begin{document}

\begin{frontmatter}


\title{On the Ext algebras of parabolic Verma modules and $A_\infty$-structures}

\author[AK,CS]{Angela Klamt and Catharina Stroppel}
\address[AK]{Department of Mathematics, Universitetsparken 5, 2100 Copenhagen  (Denmark);
email: \text{Angela.Klamt@math.ku.dk}}
\address[CS]{Department of Mathematics, University of Bonn, Endenicher Allee 60, 53115 Bonn
(Germany); email: \text{stroppel@math.uni-bonn.de}}

\begin{abstract}
We study the Ext-algebra of the direct sum of all parabolic Verma modules in the principal block of the
Bernstein-Gelfand-Gelfand category $\mathcal{O}$ for the hermitian symmetric pair
$(\mathfrak{gl}_{n+m},\mathfrak{gl}_{n}\oplus \mathfrak{gl}_m)$ and present the corresponding quiver with relations for the
cases $n=1, 2$. The Kazhdan-Lusztig combinatorics is used to deduce a general vanishing result for the higher multiplications
in the $A_\infty$-structure of a minimal model. An explicit calculations of the higher multiplications with
non-vanishing $m_3$ is included.
\end{abstract}


\end{frontmatter}





\section*{Introduction}
In 1988 Shelton determined inductively the graded dimension of the spaces of extensions $\Ext^k(M(\la), M(\mu))=\bigoplus_{k\geq0}\Ext^k(M(\la), M(\mu))$ of parabolic Verma modules
$M(\la)$ and $M(\mu)$ in the parabolic category $\O^\p$ for the Hermitian symmetric cases \cite{Shel1988}. More recently
Biagioli
 reformulated the result combinatorially and obtained a closed dimension formula \cite{Biag2004}. A nice feature is the fact that (parabolic) Verma modules form an exceptional sequence; i.e. they are labeled by a partially ordered
 set $(\La, \leq)$ of highest weights such that for all $k\geq 0$ the following holds:
$$\hom(M(\la), M(\la))=\C \text{ and } \Ext^k(M(\la),M(\mu))=0 \text{ unless } \la \leq \mu.$$
A priori the set $\La$ is infinite, but the category $\O^\p$ decomposes into indecomposable summands, so-called blocks, each
containing only finitely many of the parabolic Verma modules. Taking $M$ to be the direct sum of those
which appear in the principal block yields a finite dimensional vector space $\Ext(M,M)$ which decomposes as the direct sum of
$e_\mu \Ext(M, M) e_\la =  \Ext(M(\mu), M(\la))$, where $e_\mu$ is the projection onto $M(\mu)$ along the sum of the other
direct factors of $M$. It comes along with a natural algebra structure (the Yoneda product) which can be obtained by viewing
$\Ext(M,M)$ as the homology of the algebra $\hom(P_\bullet, P_\bullet)$ with $P_\bullet$ a projective resolution of $M$; the
multiplication is given by the composition of maps between complexes.  The construction of these projective resolutions and
chain maps requires quite detailed knowledge of the projective modules and morphisms between them. Note that already the
question about non-vanishing $\hom$-spaces between parabolic Verma modules is non-trivial (cf. \cite{Boe85} or \cite[Theorem
9.10]{Hump08}). The aim of this paper is to explore this Ext-algebra in more detail for the Hermitian symmetric case of
$(\mathfrak{gl}_{m+n},\mathfrak{gl}_{m}\oplus \mathfrak{gl}_n)$.
 In \cite{brun32008} Brundan and the second author developed a combinatorial description of the category $\O^\mathfrak{p}$
 for $\mathfrak{g}=\mathfrak{gl}_{m+n}$ and $\mathfrak{p}$ the parabolic subalgebra with Levi component $\mathfrak{gl}_m
 \oplus \mathfrak{gl}_n$ via a slight generalization of Khovanov's diagram algebra (cf. Theorem \ref{Cor:equcat}). Using
 these combinatorial techniques along with classical Lie theoretical results, provides enough tools to compute projective
 resolutions and their morphisms. As a crucial tool and byproduct we obtain a version of the
 Delorme-Schmid theorem (cf. \cite{Delo77}, \cite{schm81}) in our situation.
 The main results of the first part of the paper are Theorems \ref{mainalg1} and \ref{mainalg2}, which give an explicit
 description of the $\Ext$-algebra in terms of a path algebra of a quiver with relations for the cases $n=1$ and $n=2$,
 respectively.
The first algebra also occurs in the context of knot Floer Homology, \cite{khovanov2002quivers}, see also \cite{Asae2008}. For a connection to sutured Floer homology we refer to
\cite{Grig2010}.

In the context of Fukaya categories these algebras come along with a natural $A_\infty$-algebra structure
which encodes more information about the object. An $A_\infty$-algebra, also known in topology as a strongly homotopic
associative algebra, has higher multiplications satisfying so-called Stasheff relations (cf. \cite{Kell2001}). As Keller for instance points out, working with minimal models provides the possibility to recover the algebra of complexes filtered by a
family of modules $M(i)$ from some $A_\infty$-structure on $\Ext(\bigoplus M(i), \bigoplus M(i))$. This $A_\infty$-structure
is constructed in the form of a minimal model, i.e. deduced from an algebra structure on $H^*(\hom(\bigoplus P(i)_{\bullet},
\bigoplus P(i)_{\bullet}))$. In particular, there is a natural $A_\infty$-structure on our space of extensions $\Ext(M, M)$.
Since the projective objects are filtered by parabolic Verma modules and therefore parabolic Verma modules generate the
bounded derived category $D^b(\O^\mathfrak{p})$ it is of interest to know more about these $A_\infty$-structures. In the
second part of the paper we construct an explicit minimal model for our $\Ext$-algebra from above. The results from the first
part, in particular the explicit construction of projective resolutions, allow us to analyse the higher multiplications. For
the construction of the minimal models we mimic the approach of \cite{Lu2009} and combine formulas obtained by Merkulov
\cite{Merk99} (for the case of superalgebras) and Kontsevich and Soibelman \cite{Kont2001} (for the $\F_2$-case).  As for the
Ext-algebra structure itself we keep track of all the signs (which sometimes leads to tedious computations). Using these
techniques, we achieve the first vanishing theorem (Theorem \ref{Th:1stvanish}) in case $n=1$. In this theorem we get the
formality of the $\Ext$-algebra, i.e. we construct a minimal model with vanishing $m_k$ for $k \geq 3$. For $n=2$, in the
second vanishing theorem (Theorem \ref{Th:2ndvanish}) we have an $A_\infty$-structure with non-vanishing $m_3$ but vanishing
$m_k$ for $k \geq 4$. Thus, we obtain an example of an $A_\infty$-algebra with non-trivial higher multiplications. The
main result of the paper is presented in the general vanishing theorem (Theorem \ref{Th:genvanish}). It says that for
arbitrary $n$ we get a minimal model with vanishing $m_k$ for $k \geq n^2+2$. A crucial ingredient in the proof is a detailed
analysis of the Kazhdan-Lusztig polynomials forcing higher multiplications to vanish. This article is based on \cite{Klam10}
and focuses on presenting the main results and techniques. Some of the (very) technical detailed calculations are therefore
omitted, but can be found in \cite{Klam10}.
\subsubsection*{Acknowledgment}
The authors thank Bernhard Keller for helpful discussions and the referees for several extremely useful and detailed
comments.

\section{Preliminaries and Category $\mathcal{O}^\mathfrak{p}$}\label{sec:Cat}

We first recall the definition of the Bernstein-Gelfand-Gelfand category $\O$. For a more detailed treatment see \cite{Hump08}, \cite{Mood95}.

Let $\fr{g}$ be a finite dimensional reductive Lie algebra over $\C$ and $\fr{h} \subset \fr{b} \subset \fr{g}$ fixed
Cartan and Borel subalgebras. Denote by $\Phi \subset \fr{h}^*$ the root system of $\fr{g}$ relative to
$\fr{h}$  with the sets $\Delta \subset \Phi ^+ \subset \Phi$ of simple and positive roots respectively. For $\alpha\in\Phi$ we have the root space $\fr{g}_\al$ and the coroot $\alch\in\fr{h}$ normalized by $\al(\alch)=2$.
Let $\fr{g}=\fr{n}^- \oplus \fr{h} \oplus \fr{n}^+$ be the triangular decomposition into negative roots spaces, Cartan subalgebra and positive root spaces. Denote $\La^+ := \{ \la \in \fr{h}^*| \langle \la, \alch \rangle \geq 0 \text{ for all } \al \in \Phi^+ \}$, the set of dominant weights.

Denote by $\rho=\frac{1}{2}\sum_{\al \in \Phi^+} \al$ the half-sum of positive roots and by $\la_0$ the zero weight. Let $W$ be the \textit{Weyl group} with its usual \textit{length function} $w\mapsto l(w)$ of taking the length of a reduced expression. We get a natural action of $W$ on $\fr{h}^*$ with fixed point zero.
Shifting this fixed point to $-\rho$ defines the \textit{dot-action} $w \cdot \la=w(\la+\rho)-\rho$.
where $w \in W, \la \in \fr{h}^*$.

For $L$ any Lie algebra we denote by $U(L)$ the universal enveloping algebra. For $\la \in \fr{h}^*$ and $M$ an arbitrary $U(\fr{g})$-module the \textit{weight
space} of weight $\la$ relative to the action of the Cartan subalgebra $\fr{h}$ is defined as
$$M_\la:= \{v \in M \mid h \cdot v=\la(h)v ,\ \forall \ h \: \in \: \fr{h}\}.$$
We denote by $U(\fr{g})-\Mod$ the category of left $U(\fr{g})$-modules.\\

We fix now a standard parabolic subalgebra $\p$ containing $\fr{b}$. This corresponds to a choice of
a subset $J \subset \Delta$ with associated root system $\Phi_J \subset \Phi$ such that $\p=\fr{l}_J \oplus \fr{u}_J$ with nilradical  $\fr{u}_J$ and Levi subalgebra $\fr{l}_J=\fr{h}\oplus_{\alpha\in \Phi_J}\fr{g}_\al$.

In particular, the choice  $\p= \fr{b}$ corresponds to
$J=\emptyset$ and $\fr{l}_J=\fr{h}$, whereas $\p=\fr{g}$ corresponds to $J=\Delta$ and  $\fr{l}_J=\fr{g}$.
Let $W_\p$ be the Weyl group generated by all $\al \in J$. Denote
by $W^\p$ the set of minimal-length coset representatives for $W_\p \backslash W$, that is
$$W^\p=\{ w \in W |\  \forall \  \al \in J : l(s_\al w)>l(w)\}.$$
Define the set of {\it $\p$-dominant weights} as
\begin{equation*}
\La_J^+:=\{\la \in \fr{h}^* | \langle \la, \alch \rangle \in \Z^+ \text{ for all } \al \in J
\}. \end{equation*} Denote by $E(\la)$ the finite dimensional $\fr{l}_J$-module with highest weight $\la \in \La_J^+$.

\begin{Def}
The category $\O^\p$ is the full subcategory of $U(\fr{g})-\Mod$  whose objects $M$ satisfy the following
conditions:
\begin{itemize}
		\item[$\O 1)$] $M$ is a finitely generated $U(\fr{g})$-module; 		\item[$\O 2)$] $M$ is $\fr{h}$-semisimple,
i.e., $M=\bigoplus_{\la \in 		 \fr{h}^*} M_\la$; 		\item[$\O 3)$] $M$ is
locally $\p$-finite, i.e. $\operatorname{dim}_\mathbb{C} U(\fr{p}) \cdot v<\infty$ for all $v \in M$.
\end{itemize}
\end{Def}

We recall a few standard results on  $\O^\p$, see \cite{Hump08}, \cite{RC} for details.
\begin{Def}
For $\la \in \La_J^+$ we define the \textit{parabolic Verma module}
$$M(\la):= U(\fr{g}) \otimes_{U(\fr{p}_J)} E(\la).$$
\end{Def}
It has highest weight $\la$ and is the largest quotient contained in $\O^\p$ of the ordinary Verma module with highest weight
$\la$. In particular, it has a unique simple quotient which is denoted by $L(\la)$. The $L(\la)$, for $\la\in\La_J^+$
constitute a complete set of non-isomorphic simple objects in $\O^\p$. The category $\O^\p$ has enough projective
objects; for $\la \in \La_J^+$ let $P(\la)$ be the projective cover of $L(\la)$. The category $\O^\p$ splits into direct summands (so-called `blocks')  $\O_\chi^\p$,
$$\O^\p = \bigoplus_\chi \O_\chi^\p,$$
indexed by $W$-orbits under the dot-action. The summand $\O_\chi^\p$ is the full subcategory of modules containing only composition factors of the form $L(\la)$ with $\la \in \chi\cap \La_J^+$.
In particular $M(\la)$ and $P(\la)$ are objects of $\O_\chi^\p$ for $\la\in\chi$. Let $\O^\p_0$ be the {\it principal block} of $\O^\p$ corresponding to the orbit through zero which has precisely the $L(w \cdot {\la_0})$ with $w \in W^\p$ as simple objects. Since we work with left
cosets, for better readability we write $P(x \cdot \la)=:P(\la.x)$; similarly for simple modules and parabolic Verma modules.
\begin{Remark}{\rm
To combine later on Lie-theoretical results for $\O^\p_0(\fr{sl}_{m+n})$ with combinatorial results known for
$\O^{\p'}_0(\fr{gl}_{m+n})$ we will tacitly use the standard equivalence of categories $\O^{\p'}_0(\fr{gl}_{m+n})\cong \O^\p_0(\fr{sl}_{m+n})$
where $\p'$ is the parabolic subalgebra with corresponding Levi component $\fr{gl}_m \oplus \fr{gl}_n$ and $\p=\p' \cap
\fr{sl}_{m+n}$. }
\end{Remark}

\section{The Ext algebra}\label{ch:NotHom}

We first introduce the homological and internal shift functors, $[i]$ and $\langle i \rangle$ for $i\in\mathbb{Z}$, on the
category of complexes:
\begin{Convention}\label{Not:sign}
For a complex $C_\bullet=(C_\bullet,d_\bullet)$ define $C[i]_\bullet$ by $C[i]_j=C_{j-i}$ with
differential $d[i]_j=(-1)^i d_j$. For $M$ a graded $A$-module define the internal shift $M\langle i \rangle$ by $M\langle i
\rangle_j=M_{j-i}$. We denote by $C_\bullet\langle i \rangle$ the (internally) shifted complex $C_\bullet$ obtained
by just shifting each object; the differential maps stay homogeneous of degree
zero.
\end{Convention}
Let $A, B \in Ob(\A)$ be objects in an abelian category $\A$ and assume that $A$ and $B$ have finite projective dimension.
Given projective resolutions $P_\bullet$ and $Q_\bullet$ of $A$ and $B$, respectively, we define a differential graded
structure on $\hom(P_\bullet,Q_\bullet)$ with $\hom(P_\bullet,Q_\bullet)^r=\prod_p \hom(P_p,Q_{p+r})$ and differential
$d_p(f)=d {\circ}f-(-1)^p f {\circ}d$ (c.f. \cite[III.6.13]{Gelf88}). The space of extensions $\Ext$ can then be
computed using the derived category,
\begin{align*}
\Ext^k(A,B)&=\hom_{\D(\A)}(A[0],B[k])                &=&	\hom_{\D(\A)}(P_\bullet[0],Q_\bullet[k]) \\
           &=\hom_{\K(\A)}(P_\bullet[0],Q_\bullet[k])& =& \hom_{\K(\A)}(P_\bullet,Q_\bullet)[k]\\
           &=H^0(\hom(P_\bullet,Q_\bullet)[k]) &=& H^k(\hom(P_\bullet,Q_\bullet)),
\end{align*}
where the third equality holds because $P_\bullet$ is a bounded complex of projectives. In other words, $\Ext^k(A,B)$ can be determined by first computing the homomorphism spaces of the projective resolutions and
afterwards taking its cohomology. Cycles in $\hom(P_\bullet,Q_\bullet)$ are chain maps (according to the degree commuting or
anticommuting) and boundaries are homotopies (up to sign). If considered as chain maps between translated complexes (i.e. in
$\hom_{\D^b(\A)}(P_\bullet[0],Q_\bullet[k])$) with the sign convention \ref{Not:sign}, the cycles become commuting chain maps
and boundaries stay usual homotopies.

We are now interested in the case $A=B$ and the algebra $\Ext^k(A,A)=H^k(\hom(P_\bullet,P_\bullet))$. The multiplication in
$\Ext(A,A)$ is induced from the multiplication in the algebra $\hom(P_\bullet, P_\bullet)$, where it is given by composing of chain maps. Multiplication will be written from left to right, i.e. for $\al$, $\be \in \hom(P_\bullet, P_\bullet)$ we have $(\al \cdot \be)(x)=\be(\al(x))$.

If
$A=\bigoplus\limits_{\al \in I} A_\al$ and $P_{\al \bullet}$ is a projective resolution of $A_\al$ with corresponding
decomposition $P_\bullet=\bigoplus\limits_{\al \in I} P_{\al \bullet}$ then $\Id_\al=[\id]\in \Ext^0(A_\al,A_\al)$. The
elements $\Id_\al$ form a system of mutual orthogonal idempotents, hence we can write
$$\Ext^k(A,A)=\bigoplus_{\al, \beta \in I} \Id_\al \Ext^k(A_\al, A_\beta) \Id_\beta.$$
It is then enough to determine $\Ext^k(A_\al, A_\beta)$ for any $k$, $\alpha$, $\beta$ and the products of elements $x \in
\Ext^k(A_\al, A_\beta)$ and $y \in \Ext^l(A_\beta, A_\ga)$, interpreting their product as
$$x \cdot y =\Id_\al x\Id_\beta \Id_\beta y \Id_\gamma \in \Ext^{k+l}(A, A).$$
\section{$\O^\p(\fr{gl}_{m+n}(\C))$ via Khovanov's diagram algebra}\label{ch:Khov}
We specialize now our setup to $\fr{g}=\fr{gl}_{m+n}(\C)$ with the standard Borel subalgebra $\fr{b}$ given by upper
triangular matrices containing the Cartan ${\fr h}$ of diagonal matrices. Let $\p$ be the parabolic subalgebra associated to
the Levi subalgebra $\fr{l}=\fr{gl}_{m}(\C) \oplus \fr{gl}_{n}(\C)$. Then our key tool is the following special case of the main theorem from \cite{brun32008}, first observed in \cite{StrSpringer}:
\begin{Theorem}\label{Cor:equcat}
There is an equivalence of categories from the principal block of $\O^\p$ to the category of finite dimensional left modules
over the Khovanov diagram algebra, $K_m^n-\mod$, sending the simple module $L(\la) \in \O^\p$ to the simple module $L(\la) \in
K_m^n-\mod$, the parabolic Verma module $M(\la)$ to the cell module $M(\la)$ and the indecomposable projectives to the
corresponding indecomposable projectives.
\end{Theorem}
Here $K^n_m$ is the algebra defined diagrammatically in \cite{brun32008} with an explicit distinguished basis given by certain diagrams (see below) and a multiplication defined by an explicit ``surgery'' construction which can be expressed in terms of an extended 2-dimensional TQFT construction, \cite{StrSpringer}, generalizing a construction of Khovanov \cite{KhovJones}. The distinguished basis is in fact a (graded) cellular basis in the sense of Graham and Lehrer \cite{Grah2004} in the graded version of Hu and Mathas \cite{Hu2010}. The algebra is shown to be quasi-hereditary in \cite{Brun12008}. Hence we have  cell or standard modules $M(\la)$, their projective covers $P(\la)$ and irreducible quotients $L(\la)$. This is meant by the notation used in the theorem.

\subsection{The algebra $K_m^n$ and its basic properties}
For the construction of $K_m^n$, we recall from \cite{Brun12008} the notions of weights, cup/cap/circle diagrams adapted to our situation. Let $\la\in\La_J^+$ be the highest weight of a simple module in $\O^\p_0=\O^\p(\fr{gl}_{m+n}(\C)_0$ and let
$$\rho=\eps_{m+n-1}+2\eps_{m+n-2}+ \cdots +(m+n-1) \eps_{1} \in \fr{h}^*.$$
The {\it (diagrammatical) weight} associated to $\la$ is obtained by labeling the number $i$ on the real line by $\down$ if $i$ belongs to $I_\down(\la)$ and by $\up$ if $i$ belongs to $I_\up(\la)$ respectively, where
\begin{eqnarray*}
I_\down(\la)&:= &\{(\la+\rho,\eps_1), \ldots, (\la+\rho,\eps_m)\}\\
I_\up(\la)&:=&\{(\la+\rho,\eps_{m+1}), \ldots,(\la+\rho,\eps_{m+n})\}.
\end{eqnarray*}
\begin{figure}
 \center
\begin{tikzpicture}
\begin{scope}
		\draw (-1, 0)--(2.6, 0); 		\draw (0,0) node[below=-3pt] {$\up$}; 		\draw (0.4,0) node[below=-3pt] {$\up$}; 		 
\draw (0.8, 0) node[above=-3pt] {$\down$}; 		\draw (1.2,0) node[above=-3pt] {$\down$}; 		\draw (1.6,0) node[above=-3pt]
{$\down$}; 		\end{scope} 		\begin{scope}[yshift=0cm] 		\scriptsize 		\draw (-1.4, 0) node {$\cdots$}; 		 
\draw (-0.8, 0.1)--(-0.8, -0.1); 		\draw (-0.9, 0) node [above] {$-2$}; 		\draw (-0.4, 0.1)--(-0.4, -0.1); 		
\draw (-0.5,0) node[above] {$-1$}; 		\draw (2,0) node[above] {$5$}; 		\draw (2, 0.1)--(2, -0.1);			\draw (2.4,0)
node[above] {$6$}; 		\draw (2.4, 0.1)--(2.4, -0.1);				\draw (3,0) node {$\cdots$}; 		\end{scope}
\end{tikzpicture}
\caption{the zero weight for $n=2$ and $m=3$} 	\label{fig:zeroweight}
\end{figure}
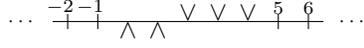
Let $\La_m^n$ be the set of diagrammatical weights obtained in this way. Note that the labels are always on the $(m+n)$ places $i \in I=\{0,\dots,m+n-1\}$ which we call {\it vertices}. The diagrammatical weight associated to $\la_0$ is given by putting all $\up$'s to the left and all $\down$'s to the right, see Figure \ref{fig:zeroweight}. In fact, $\La_m^n$ consists precisely of the diagrams obtained by permuting the $n \ \up$'s and $m \  \down$'s establishing a bijection between the highest weights of parabolic Verma modules in $\O^\p_0$ and elements in $\La_m^n$. The dot-action  corresponds then to permuting the labels; swapping $\down$'s to the right means getting bigger in the Bruhat order, see \cite[Section 1]{brun32008}.

We fix the above bijection and do not distinguish in notation between weights and diagrammatical weights.
For $\la=\la_0.x$ with $x \in W^\p$ we write $l(\la)$ for $l(x)$. For each $i\in I$ define \emph{the relative length}
\begin{eqnarray}
\label{rell}
l_i(\la,\mu) :=& \#\{j \in I\:|\:j \leq i \text{ and vertex $j$ of $\la$ is labeled $\down$}\}\nonumber
\\
&-
\#\{j \in I\:|\:j \leq i \text{ and vertex $j$ of $\mu$ is labeled $\down$}\}
\end{eqnarray}
and note that $l(\la)-l(\mu)=\sum_{i \in I} \ell_i(\la,\mu)$ by \cite[Section 5]{Brun12008}.

A {\it cup diagram} is a diagram obtained by attaching rays and finitely many cups (lower semicircles) to the vertices $I$, so that cups join two vertices $i\in I$,  rays join vertices $i\in I$ down to infinity, and rays or cups do not intersect other rays or cups. A {\it cap diagram} is the horizontal mirror image of a cup diagram, so caps (i.e. upper semicircles) instead of cups are used. The mirror image of a cup
 (resp. cap) diagram $c$ is denoted by $c^*$.\\

If $c$ is a cup diagram and $\la$ a weight in $\La_m^n$, we can glue $c$ and $\la$ and obtain a new diagram denoted $c\la$. It is called an {\it oriented cup diagram} if
\begin{itemize}
	\item each cup is oriented, i.e. one of its vertices is labeled $\down$, and one $\up$; 	\item there are
not two rays in $c$ labeled $\down \up$ in this order from left to right.
\end{itemize}
An example is given in Figure \ref{fig:anOrientedDiagram}.
\begin{figure}
 \center
\begin{tikzpicture}
\begin{scope}
		\draw (0,0) node[above=-1.55pt] {$\up$}; 		\draw (0.4,0) node[above=-1.55pt] {$\down$}; 		\draw (0.8, 0)
node[above=-1.55pt] {$\down$}; 		\draw (1.2,0) node[above=-1.55pt] {$\up$}; 		\draw (1.6,0) node[above=-1.55pt]
{$\down$}; 		\draw (2,0) node[above=-1.55pt] {$\up$}; 		\draw (2.4, 0) node[above=-1.55pt] {$\down$}; 		\draw
(0,0) -- (0,-0.6); 		\draw (0.4,0) -- (0.4,-0.6); 	  \draw (0.8,0) arc (180:360:0.6); 		\draw (1.2,0) arc
(180:360:0.2); 		\draw (2.4,0) -- (2.4,-0.6);			\end{scope}
\end{tikzpicture}
\hspace{2cm}
\begin{tikzpicture}
\begin{scope}
		\draw (0,0) node[above=-1.55pt] {$\up$}; 		\draw (0.4,0) node[above=-1.55pt] {$\down$}; 		\draw (0.8, 0)
node[above=-1.55pt] {$\down$}; 		\draw (1.2,0) node[above=-1.55pt] {$\up$}; 		\draw (1.6,0) node[above=-1.55pt]
{$\down$}; 		\draw (2,0) node[above=-1.55pt] {$\up$}; 		\draw (2.4, 0) node[above=-1.55pt] {$\down$}; 		\draw
(2.4,0.4) -- (2.4,0.8); 		\draw (0,0.4) arc (180:0:0.2); 		\draw (0.8,0.4) arc (180:0:0.2); 		\draw (1.6,0.4)
arc (180:0:0.2); 		\draw (0,0) -- (0,-0.6); 		\draw (0.4,0) -- (0.4,-0.6); 	  \draw (0.8,0) arc (180:360:0.6); 		 \draw (1.2,0) arc (180:360:0.2); 		\draw (2.4,0) -- (2.4,-0.6);			\end{scope}
\end{tikzpicture}
\caption{An oriented cup diagram and an oriented circle diagram.} 	\label{fig:anOrientedDiagram}
\end{figure}
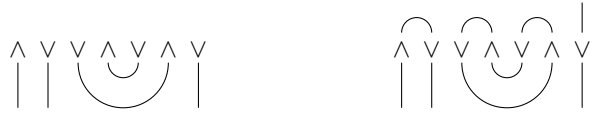

Similarly we can glue $\la$ to a cap diagram $c$. The result  $\la c$ is called {\it oriented cap diagram} if $c^*\la$ is an oriented cup diagram. A {\it circle
diagram} is obtained by gluing a cup and a cap diagram at the vertices $I$. It consists of circles and lines. Gluing an oriented cap diagram with an oriented cup diagram along the same weight gives an {\it oriented circle diagram}.
For an example, see Figure \ref{fig:anOrientedDiagram}.

The {\it degree} of an oriented cup/cap diagram $a\la$ (or $\la b$) means the total number of oriented cups (caps) that it
contains. So in $K_m^n$ one has $deg(a \la) \leq n$, since there are at most $n$ cups. The {\it degree} of an oriented
circle diagram $a\la b$ is defined as the sum of the degree of $a\la$ and the degree of $\la b$. The {\it cup diagram
associated to a weight $\la$} is the unique cup diagram $\underline{\la}$ such that $\underline{\la} \la$ is an oriented cup
diagram of degree 0. (For an explicit construction: Take any two neighboring vertices labeled by $\down \up$ and connect
them by a cup. Repeat this procedure as long as possible, ignoring vertices which are already joined to others. Finally draw
rays to all vertices which are left.) The {\it cap diagram associated to a weight $\la$} is defined as
$\overline{\la}:=(\underline{\la})^*$.
The vector space underlying $K_m^n$ has a basis
\begin{eqnarray*}
\left\{(a\la b) \left| \text{for all oriented circle diagrams with } \la \in \La_m^n\right\} \right. .
\end{eqnarray*}
Each basis vector has a well-defined degree, turning the vector space into a graded vector space equipped with a distinguished homogeneous basis. The element $e_\la$ is defined to be the diagram $\underline{\la}\la \overline{\la}$. The
product of two circle diagrams $a\la b$ and $c \mu d$ is zero except for $b=c^*$. The multiplication of $a \la b$ and $b^* \mu
d$ works by the rules of the generalized surgery procedure defined in \cite[Section 3 and Theorem 6.1.]{Brun12008}. The
vectors $\left\{ e_\alpha | \alpha \in \La_m^n \right\} $ form a complete set of mutually orthogonal idempotents in $K_m^n$.
We get
\begin{eqnarray*}
K_m^n&=&\bigoplus_{\alpha, \beta \in \La_m^n} e_\alpha K_m^n e_\beta
\end{eqnarray*}
where $e_\alpha K_m^n e_\beta$ has basis $\left\{(\und{\alpha} \la \ove{\beta}) \left| \lambda \in \La_m^n \text{ such that the diagram is oriented} \right\} \right..$

\subsection{Modules}
Theorem \ref{Cor:equcat} establishes an equivalence of categories between $\O^\p_0$ and the category of finite dimensional $K_m^n$-modules. Following \cite{Brun12008}, we consider the category $K_m^n-\gmod$ of finite dimensional {\it graded} left $K_m^n$-modules which can be seen as a graded version of $\O^\p_0$ with the following important objects:
\begin{itemize}
	\item The {\it simple modules} $L(\la)$ with $\la \in \La_m^n$. \\ 	These are $1$-dimensional modules concentrated in
degree zero. The idempotent $e_\la \in  K_m^n$ acts by the identity, all other $e_\mu$ by zero. Shifting the internal degree gives all simple objects, $L(\la)\langle i\rangle$, $i\in\mathbb{Z}$. 	
\item The {\it projective cover}
$P(\la)=K_m^n e_\la$ of the simple module $L(\la)$ has homogeneous basis
\begin{align*} 	\left\{ (\und{\alpha}\mu\ove{\la}) \left| \text{
for all } \alpha, \mu \in \La_m^n \text{ such that the diagram is oriented}\right\}; \right.\end{align*} 	with the
action induced from the diagrammatical multiplication in the algebra. By shifting the internal degree one obtains a full
set of indecomposable graded projective modules. 	\item The {\it cell or standard modules} $M(\mu)$ with homogeneous
basis 		
\begin{align*} 	\left\{ (c \mu | \:\big|\: \text{for all oriented cup diagrams $c
\mu$}\right\} 		\end{align*} such that $(a\la b)(c\mu |)=(a\mu |)$ or $0$ depending on the elements.
\end{itemize}
After forgetting the grading, these modules correspond via Corollary \ref{Cor:equcat} to simple modules, projectives and Verma modules
in the principal block of $O^\p$.
\subsection{q-decomposition numbers}
We have the following theorems about cell module filtrations of projectives and Jordan-H\"older filtrations of cell modules,
which say that $K_m^n$ is quasi-hereditary in the sense of Cline, Parshall and Scott \cite{Clin1988}.
\begin{Lemma}[{\cite[Theorem 5.1]{Brun12008}}]
\label{qh1} For $\la \in \La_m^n$, enumerate the elements of the set $\{\mu \in \La_m^n\:|\:\und{\la}\mu \text{ is
oriented}\}$ as $\mu_1,\mu_2,\dots,\mu_n = \la$ so that $\mu_i > \mu_j$ implies $i < j$. Let $M(0) := \{0\}$ and for
$i=1,\dots,n$ define $M(i)$ to be the subspace of $P(\la)$ generated by $M(i-1)$ and the vectors
$$
\left\{ (c \mu_i \overline{\la} ) \:\big|\: \text{for all oriented cup diagrams $c \mu_i$}\right\}.
$$
Then
$
M(0) \subset M(1) \subset\cdots\subset M(n) = P(\la)
$
is a filtration of $P(\la)$ as graded $K_m^n$-module such that
$M(i) / M(i-1) \cong M(\mu_i) \langle \deg(\mu_i \overline{\la})\rangle$
for $1\leq i\leq n$.
\end{Lemma}
\begin{Lemma}[{\cite[Theorem 5.2]{Brun12008}}]\label{qh2}
For $\mu \in \La_m^n$, let $N(j)$ be the submodule of $M(\mu)$ spanned by all graded pieces of degree $\geq j$. This defines a finite filtration of the graded $K_m^n$-module $M(\mu)$ with simple subquotients
$$
N(j) / N(j+1) \cong \bigoplus_{\substack{\la  \subset \mu\text{\,with}\\ \deg(\underline{\la} \mu) = j}}
  L(\la) \langle j \rangle.
$$
\end{Lemma}
By the BGG-reciprocity \cite[Theorem 9.8(f)]{Hump08} the two multiplicities $d_{\la,\mu}^i:=[M(\mu):L(\la)\langle i\rangle]$
and $[P(\la):M(\mu)\langle i \rangle]$ are equal and we get the symmetric \textit{q-Cartan matrix}
$$C_{\La_m^n}(q) = (c_{\la,\mu}(q))_{\la,\mu \in\La_m^n},$$
where
$$c_{\la,\mu}(q) := \sum_{j \in \Z} \dim \hom_{K_m^n}(P(\la),P(\mu))_j\; q^j \in \Z[q].$$
Set $d_{\la,\mu}(q)=\sum_{i}d^i_{\la,\mu}q^i$. Note that this sum in fact contains at most one non-trivial summand, since
$d_{\la,\mu}\not=0$ implies $\und{\la}\mu$ is oriented and $\la\leq\mu$ in the Bruhat ordering, in which case
$d_{\la,\mu}=q^{\operatorname{deg}(\und{\la}\mu)}$ holds (cf. \cite[5.12]{Brun12008}).

In a cup (cap) diagram we number the cups (caps) $1,2,\ldots$ according to their right vertex from left two right. For a cup (cap) diagram $a$ we denote by $\nes_a(i)$ for $1 \leq i \leq \# \{ \text{cups} \}$ the
number of cups nested in the $i$th cup.

The following provides then explicit
lower and upper bounds for the decomposition numbers and the entries of the $q$-Cartan matrix:
\begin{Prop}
\label{Le:condmaps} In $K_m^n-\gmod$ we have $d_{\la,\mu}=0$ unless
\begin{equation}
\label{inequ1} 0 \leq l(\la)-l(\mu) \leq n+2\sum_i \nes_{\und{\la}}(i) \leq n^2.
\end{equation}
In particular, $c_{\la, \mu}=0$ unless $l(\la)-l(\mu) \leq n+2\sum_i \nes_{\und{\la}}(i) \leq n^2.$
\end{Prop}
\begin{proof}
Assume $d_{\la,\mu}(q)\neq 0$. This means that $\und{\la}\mu$ is oriented. By \cite[Lemma 2.3]{Brun12008} it follows that $\la
\leq \mu$ in the Bruhat ordering, which leads to $l(\la) \geq l(\mu)$. Now we find $\la$ and $\mu$ such that $l(\la)-l(\mu)$
is maximal and $\und{\la} \mu$ is oriented. Fix such $\la$ and consider weights $\mu$ of smallest possible length such that
$\und{\la}\mu$ is still oriented. This is obtained if all $\up$'s and $\down$'s on the end of a cup in $\la$ are interchanged.
Since a $\up$ on the $i$th cup has been moved $1+2 \nes_{\und{\la}}(i)$ positions to the right, the length is changed by
$\sum_i(2 \nes_{\und{\la}}(i)+1)$. Therefore, we obtain
$$0 \leq l(\la)-l(\mu) \leq n+2\sum_i \nes_{\und{\la}}(i).$$
Since $\sum_i \nes_{a}(i)$ is maximal if all cups are nested (i.e if the $j$th cup contains precisely $j-1$ cups). In that
case we obtain
$$2\sum_i \nes_{a}(i)=2\sum_{i=1}^n (i-1)=(n-1)n$$ and therefore \eqref{inequ1} holds.
For $c_{\la, \mu} \neq 0$ a simple $L(\la)$ must occur in $P(\mu)$, especially it must occur in some $M(\nu)$, i.e.
$d_{\la,\nu}\neq 0$ and $d_{\mu,\nu}\neq 0$. Therefore,
$$l(\la)-l(\nu) \leq n+2\sum_i \nes_{\und{\la}}(i)$$
and $0 \leq l(\mu)-l(\nu)$ which implies
$$l(\la)-l(\mu) \leq l(\la)-l(\nu) \leq n+2\sum_i \nes_{\und{\la}}(i),$$
which proves the second inequality.
\end{proof}
\subsection{Linear projective resolutions of cell modules}\label{End}
To compute the Ext-algebras of Verma modules it will be useful to construct explicitly linear projective resolutions of the cell modules $M(\la)\in K_m^n-\gmod$. Recall that a projective resolution $P_\bullet$ is {\it linear} if $P_i$ is generated by its homogeneous component in degree $i$. From the description of projective modules it is clear that
$\bigoplus_{\la\in\La_m^n} P(\la)\cong K_m^n$ is a minimal projective generator of $K_m^n-\mod$.
Any endomorphism is given by right multiplication with an element of the
algebra, and
$\hom_{K_m^n}(P(\la), P(\mu)) = \hom_{K_m^n} (K_m^n e_\la, K_m^n e_\mu)= e_\la K_m^n e_\mu$
as vector spaces, \cite[(5.9)]{Brun12008}.

To construct the differentials in linear projective resolutions, we study first the degree $1$ component of $\hom_{K_m^n}(P(\la), P(\mu))$, i.e. we search for elements $\nu$ s.t.
$\deg(\und{\la}\nu\ove{\mu})=1$. Since $1=\deg(\und{\la}\nu\ove{\mu})=\deg(\und{\la}\nu)+\deg(\nu\ove{\mu})$, one summand
has to be $0$ and the other $1$.
\begin{enumerate}
	\item $\deg(\und{\la}\nu)=0$, i.e. $\la=\nu$, so we look for an oriented cap diagram $\la\ove{\mu}$ of degree $1$. It
exists iff $\la>\mu$ and $\mu= \la.w$ with $w$ changing the $\up$ and $\down$ (in this ordering) at the end of a cup into
a $\down$ and $\up$. 		\item $\deg(\nu \ove{\mu})=0$, i.e. $\mu=\nu$, so we look for an oriented cup diagram
$\und{\la}\mu$ of degree 1. It exists iff $\mu>\la$ and $\la=  \mu.w$ with $w$ changing the $\down$ and $\up$ at the end
of a cap.
\end{enumerate}
Altogether we get $\dim\;\hom_{K_m^n}(P(\la), P(\mu))_1\leq 1$ and the diagram calculus defines a distinguished morphism $f_{\la,\mu}$ in case this dimension equals $1$.

On the other hand, the modules occurring in a linear projective resolution of cell modules are determined by polynomials $p_{\la,\mu}$ defined diagrammatically and recursively in \cite[Lemma 5.2.]{Brun22008}, namely certain Kazhdan-Lusztig polynomials going back to work of Lascoux and Sch\"utzenberger \cite{Lasc1981}.

We recall the construction of these polynomials. Set $p_{\la,\mu}=0$ if $\la \not\le
\mu$. A \emph{labeled cap diagram} C is a cap diagram whose unbounded chambers are labeled by zero and given two chambers
separated by a cap, the label in the inside chamber is greater than or equal to the label in the outside chamber.
\begin{Def}
Denote by $D(\la,\mu)$ the set of all labeled cap diagrams obtained by labeling the chambers of $\ove{\mu}$ in such a way that for every inner cap $c$ (a cap containing no smaller one), the label $l$ inside $c$ satisfies
$l\leq l_i(\la,\mu)$, where $i$ denotes the vertex of $c$ labeled by $\down$.  The
polynomials are given by
\begin{eqnarray}
\label{KL}
p_{\la,\mu}(q) :=\sum_i p_{\la,\mu}^{(i)}q^i:= q^{l(\la)-l(\mu)}\sum_{C \in D(\la,\mu)} q^{-2|C|}.
\end{eqnarray}
where $\left|C\right|$ denotes the sum of all labels in $C$.
\end{Def}
\begin{Example}
Figure \ref{fig:kaz} presents the possible labeled cap diagrams from $D(\la, \mu)$ for the chosen $\la$ and
$\mu$. Since $l(\la)-l(\mu)=4$, we get
$p_{\la,\mu}(q)=q^{4}+q^{2}$. 		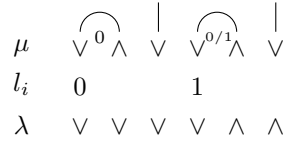
\begin{figure} 		 \center 						
 		\begin{tikzpicture}[scale=1.3] 		\draw (-0.2,0) node[above=-1.55pt] {$\mu$}; 		\draw (0.4,0)
node[above=-1.55pt] {$\down$}; 		\draw (0.8, 0) node[above=-1.55pt] {$\up$}; 		\draw (1.2,0) node[above=-1.55pt]
{$\down$}; 		\draw (1.6,0) node[above=-1.55pt] {$\down$}; 		\draw (2,0) node[above=-1.55pt] {$\up$}; 		\draw
(2.4, 0) node[above=-1.55pt] {$\down$}; 		\draw (0.8,0.3) arc (0:180:0.2); 		 \draw (1.2, 0.3) -- (1.2, 0.6); 		 
\draw (2,0.3) arc (0:180:0.2); 		\draw (2.4, 0.3) -- (2.4, 0.6); 		
\begin{scope}[yshift=-0.4cm] 		\draw (-0.2,0) node[above=-1.55pt] {$l_i$}; 		\draw (0.4,0) node[above=-1.55pt]
{$0$}; 		\draw (1.6,0) node[above=-1.55pt] {$1$}; 		\end{scope} 		\begin{scope}[yshift=0.15cm] 		
\scriptsize			\draw (0.6,0) node[above=-1.55pt] {$0$}; 		\tiny 		\draw (1.82,0) node[above=-1.55pt] {$0/1$};			 

\end{scope}			\begin{scope}[yshift=-0.8 cm] 		\draw (-0.2,0) node[above=-1.55pt] {$\la$}; 		\draw (0.4,0)
node[above=-1.55pt] {$\down$}; 		\draw (0.8,0) node[above=-1.55pt] {$\down$}; 		\draw (1.2, 0) node[above=-1.55pt]
{$\down$}; 		\draw (1.6,0) node[above=-1.55pt] {$\down$}; 		\draw (2.0,0) node[above=-1.55pt] {$\up$}; 		\draw
(2.4,0) node[above=-1.55pt] {$\up$}; 		\end{scope} 		\end{tikzpicture}
\caption{the labeled cap diagram}
\label{fig:kaz}
\end{figure}
\end{Example}

\begin{Theorem}[{\cite[Theorem 5.3]{Brun22008}}, {\cite[Theorem 3.20]{Klam10}}]
\label{Th:constrproj}
For $\la \in \La_m^n$ the cell module $M(\la)$ has a linear projective resolution $P_\bullet(\la)$ of the form
\begin{equation}
\label{profres} \cdots \stackrel{d_1}{\longrightarrow} P_{1}(\la) \stackrel{d_0}{\longrightarrow} P_{0}(\la)
\stackrel{\eps}{\longrightarrow} M(\la) \longrightarrow 0
\end{equation}
with $P_0(\la)= P(\la)$ and $P_i(\la)= \bigoplus_{\mu \in \La_m^n} p_{\la,\mu}^{(i)} P(\mu)\langle i \rangle$ for $i \geq
0$.
\end{Theorem}	

Using the above observations and tools from the proof of \cite[Theorem 5.3]{Brun22008}, \cite[§3.3.3]{Klam10} gives an explicit method to construct projective resolutions of cell modules in $K_m^n-\gmod$ by an interesting simultaneous induction
varying the underlying algebra and the highest weights. For $K_{m}^0$ and $K_0^n$ we have, up to isomorphism, only
one indecomposable module, which is projective, simple and cell module at once. This provides the starting point of the
induction. In the following we will fix such a projective resolution $P_\bullet(\la)$  for each $\la$. Together with the
inequalities obtained before, we can deduce:
\begin{Prop}\label{Le:termprojres}
If a projective module $P(\nu)$ occurs as a direct summand in $P_i(\la)$ with $P_\bullet(\la)$ being the projective resolution
constructed above, one has
\begin{align*}
l(\la)-i-\left(n^2-n-2\sum_i \nes_{\und{\nu}}(i)\right) \leq l(\nu) \leq l(\la) -i.
\end{align*}
\end{Prop}
\begin{proof}
Let $C$ be a cap connected with the $j$th $\up$
occurring in $\ove{\nu}$ and let it be the $k_j$th cup in our numbering with starting point $i$. Recall from \eqref{rell} that $l_i(\la, \nu)\leq\{k|\:k \leq i \text{ and vertex $k$ of $\nu$ is labeled $\up$}\}$, the latter
counting the numbers of $\up$'s to the left of the cap. This equals $j-1-\nes_{\und{\nu}}(k_j)$ counting to ones the left of the $j$th $\up$ without those lying inside the cap, and thus
\begin{eqnarray*}
0 \leq |C| &\leq \displaystyle{\sum_{\substack{j \in \{ 1,\ldots n\}\\ \text{cap ending on $j$th
$\up$}}}}{(j-1-\nes_{\und{\nu}}(k_j))}\leq\frac{n(n-1)}{2}-\sum_i \nes_{\und{\nu}}(i).
\end{eqnarray*}
If a module $P(\nu)$ occurs in the resolution (say at homological degree $i$), one has $p_{\la,\nu}^{(i)} >0$, i.e. there is a
diagram $C$ such that $i=l(\la)-l(\nu)-2|C|$. Taking the upper and lower bound for $C$ obtained before, one gets
$$l(\la)-i-(n^2-n-2\sum_i \nes_{\und{\nu}}(i)) \leq l(\nu) \leq l(\la) -i$$
and the claim of the proposition follows.
\end{proof}										
The following is a vanishing result for $\Ext^k(M(\la),M(\mu))$:
\begin{Lemma}\label{Le:mapprojres}
For $\la$, $\mu \in \La_m^n$ we have
\begin{align}
\hom^k(P_\bullet(\la), P_\bullet(\mu))=0 \text{ unless } l(\la) \leq l(\mu)+n^2+k. \label{inequ}
\end{align}
\end{Lemma}
\begin{proof}
A map between $P_\bullet(\la)$ and $P_\bullet(\mu)[k]$ is in each component a morphism between graded projective modules.
Including the shift we therefore have to consider morphisms between projectives $P(\nu)$ occurring in $P_i(\la)$ and
projectives $P(\nu')$ in $P_{i-k}(\mu)$. By Proposition \ref{Le:termprojres} we know
\begin{eqnarray*}
l(\la)-i-\left(n^2-n-2\sum_i \nes_{\und{\nu}}(i)\right) \leq l(\nu)
&and&
l(\nu') \leq l(\mu)-(i-k).
\end{eqnarray*}
Therefore, we have
\begin{equation}\label{1stnu}
l(\la)-i-\left( n^2-n-2\sum_i \nes_{\und{\nu}}(i)\right)-\left(l(\mu)-(i-k)\right)\leq l(\nu)-l(\nu').
\end{equation}
Since we have a morphism between these projectives we get from Lemma \ref{Le:condmaps}
\begin{equation}\label{2ndnu}
l(\nu)-l(\nu') \leq n+ 2\sum_i \nes_{\und{\nu}}(i).\end{equation} Combining the two inequalities \eqref{1stnu} and
\eqref{2ndnu}, we obtain
\begin{equation*}
l(\la)-i-\left( n^2-n-2\sum_i \nes_{\und{\nu}}(i)\right)-\left(l(\mu)-(i-k)\right) \leq n+2\sum_i
\nes_{\und{\nu}}(i),\end{equation*} which implies $l(\la) \leq l(\mu)+n^2+k$.
The claim follows.
\end{proof}															

\section{The $\Ext$-algebra of $\bigoplus_{x\in W^\p} M(\la_0\cdot x)$}\label{ch:Ext}			
Assume we are in the setup of Section \ref{ch:Khov} and denote
\begin{eqnarray*}
E_m^n=\bigoplus_{x,y\in W^\p}\Ext\left(M(x\cdot\la_0),M(y\cdot\la_0)\right)=\bigoplus_{\la,\mu\in\La_m^n}\Ext_{K_m^n}
\left(M(\la),M(\mu)\right).
\end{eqnarray*}
A very useful tool for describing $E_m^n$ are Shelton's recursive dimension formulas which he established in \cite{Shel1988} more generally for all the hermitian symmetric cases. For an arbitrary parabolic subalgebra $\p$, there is no explicit formula, not even a candidate.

Abbreviating $E^k(x,y)=\dim \Ext^k(M(\la_0.x), M(\la_0.y))$ for $x,y \in W^{\fr{p}}$, \cite[Theorem 1.3]{Shel1988} can be formulated as follows:
\begin{Theorem}[Dimension of $\Ext$-spaces]\label{Th:dim}
With $\fr{g}$ and $\p$ as above, let $x,y \in W^{\fr{p}}$ and let $s$  be a simple reflection with $x>xs$ and $xs \in
W^{\fr{p}}$. The dimensions $E^k(x,y)$ are then given by the following formulas:
\begin{align*}
1. \ &E^k(x,y)=&&0  &&\forall \ k \text{ unless } y\leq x; \\ 2. \ &E^k(x,x)=&&\begin{cases} 1 &\text{  for } k=0\\ 																			 
0 &\text{ otherwise.} \end{cases} \\ \intertext{For $y<x$ there are the following recursion formulas:} 3.\ & E^k(x,y)
=&&E^k(xs,ys) &&\text{ if }y>ys \text{ and }ys \in W^{\fr{p}}; \\ 4.\ & E^k(x,y) =&&E^{k-1}(xs,y) &&\text{ if }ys \notin
W^{\fr{p}}; \\ 5.\ & E^k(x,y) =&&E^{k-1}(xs,y) +E^k(xs,y) &&\text{ if }ys>y\text{ but }xs \not> ys; \\ 6.\ &E^k(x,y)
=&&E^{k-1}(xs,y) - E^{k+1}(xs,y)\notag \\ 	&&&+E^k(xs,ys) 	&&\text{ if }x>xs>ys>y.	 	 \end{align*}
\end{Theorem}			
To translate between our setup and Shelton's note that he denotes $N_y=M(\la_0.\omega_{\fr{m}}y \omega_0)$ where $\omega_0$ and $\omega_{\fr{m}}$ are the longest elements in $W$ and in $W_{\fr{p}}$ respectively. Then it only remains to observe that
for $y, x \in W$ we have $\omega_{\fr{m}}y \omega_0 \in W^{\fr{p}} \Leftrightarrow y \in W^{\fr{p}}$ and  $\omega_{\fr{m}}y \omega_0<\omega_{\fr{m}}x \omega_0 \Leftrightarrow y > x$ in the Bruhat order.

Although the previous theorem determines all dimension of Ext-spaces, it is convenient to have explicit vanishing
conditions. Therefore, we reprove the Delorme-Schmid Theorem (cf. \cite{Delo77}, \cite{schm81}) in our situation:
\begin{Lemma} \label{Le:Extrest}
For $\la, \mu \in \La_m^n$ we have
\begin{align*}
\Ext^k(M(\la),M(\mu))=0 \ \quad\forall \ k >l(\la) - l(\mu).
\end{align*}
\end{Lemma}
\begin{proof}
We claim that any chain map $f:P_\bullet(\la) \to P_\bullet(\mu)[k]$ with $k >l(\la) - l(\mu)$ is
homotopic to zero. On the $k$th component $f$ induces a map $f_k:P_k(\la) \to P_0(\mu)=P(\mu)$. For $P(\nu)$ occurring as a
direct summand in $P_k(\la)$ we have $l(\nu) \leq l(\la) -k <l(\la)-(l(\la)-l(\mu))=l(\mu)$ by Lemma \ref{Le:termprojres}. By
Proposition \ref{Le:condmaps} $L(\nu)$ does not occur in $M(\mu)$ and so the composition $P(\nu) \to P(\mu) \to M(\mu)$
is zero. Let $P^T_\bullet(\la)$ be the truncated complex with $P^T_i(\la) = 0$ for $i<0$ and $P^T_i(\la)=P_{i+k}(\la)$ if $i \geq 0$. This is a projective resolution of $\operatorname{im} d_k$, and $f_\bullet$ induces a morphism
$\widetilde{f}_\bullet: P^T_\bullet(\la) \to P_\bullet(\mu)$ such that
\begin{displaymath}
\xymatrix{ 0 \ar[r] &\cdots \ar[r] \ar[d] &P^T_0(\la)\ar[d]^{\widetilde{f}_0} \ar[r] & \operatorname{im} d_k \ar[r] \ar[d]^{0}
\ar[r] &0\\
 0 \ar[r] &\cdots \ar[r]  &P(\mu) \ar[r] & M(\mu) \ar[r]&0}
\end{displaymath}
where $\widetilde{f}$ is a lift of the zero map. Since the zero map between the complexes is also a lift of the zero map and two
lifts are equal up to homotopy (\cite[Theorem III.1.3]{Gelf88}) the map $\widetilde{f}$ is nullhomotopic by a homotopy
$H:P_\bullet^T(\la) \to P_\bullet(\mu)[-1]$. This extends to a homotopy $H:P_\bullet(\la) \to P_\bullet(\mu)[-1]$ by defining
it to be zero on the other terms. The claim follows.
\end{proof}
\begin{Remark}{\rm
The result of Lemma \ref{Le:Extrest} could also be deduced from Shelton's formulas or from the explicit formulas \cite[Theorem 3.4]{Biag2004}.}
\end{Remark}
\section{Special cases}\label{ch:spec}
Now we want to describe the $\Ext$-algebra in the cases $(m,n)=(1,N)$ and $(m,n)=(2,N-1)$. The first algebra is related to algebras appearing in (knot) Floer homology, see \cite{khovanov2002quivers}, \cite{Grig2010}, the second invokes our theory in a more substantial way and provides interesting $A_{\infty}$-structures.

Using knowledge about decomposition numbers, the endomorphism spaces of projective modules and the projective resolutions together with the tools worked out above, one can choose explicit maps between the projective resolutions from Theorem \ref{Th:constrproj} and determine their linear dependence up to null homotopies. In this way we will construct non-trivial elements in $\Ext^i$ which, using Shelton's dimension formulas, can be shown form a basis. Finally we compute the multiplication rules. Especially in the case for $n=2$ the computations are long and cumbersome and carried out in \cite{Klam10}. We present the crucial computations for the $n=1$ case here, which suffice in this case to get the results by a few easy straightforward calculations. For the $n=2$ case we present the results and main idea and refer to
\cite{Klam10} for the details.
\subsection{The case $n=1$}
The elements in $W^\p$ are precisely $s_1 \cdots s_j$, $0\leq j\leq N-1$ and we abbreviate $(j)=\la_0.s_1s_2\dots s_j$.
The filtrations in Theorems \ref{qh1} and \ref{qh2} combined determine the filtration of projective modules in terms of simple modules presented in Table \ref{prosimp1}.
\begin{table}
\footnotesize \caption{\label{prosimp1} Filtration of projective module $P(\la)$ by simple modules, same colour belonging to
the same Verma module} 	\centering		\begin{tabular}{|l|c|} 		\toprule 		$\la=(j)$ & $P(j)$\\ 		\hline 		
$\begin{array}{l} j \neq 0 \\ j \neq N\end{array}$& 		$\begin{array}{l} \hspace*{1cm}\blu L(j)\\ \blu L(j+1)   \re
L(j-1) \\ \hspace*{1cm} \re L(j)
\end{array}$\\ \hline
		$j=0$& 		$\begin{array}{l} \blu L(0)\\ \blu L(1)
\end{array}$\\ \hline
		$j=N$& 		$\begin{array}{l} \blu L(N)\\
 \re L(N-1) \\
\re L(N)
\end{array}$\\ \hline
		 						\end{tabular}
\end{table}
To compute the combinatorial Kazhdan-Lusztig polynomials which determine the terms of the resolution of the cell module $M(\la)$ we consider $(s)=\mu \geq \la=(j)$ and obtain
$$\begin{tikzpicture}[scale=1]
		\draw (-0.6,0) node[above=-1.55pt] {$\mu$}; 		\draw (0,0) node[above=-1.55pt] {$\cdots$}; 		\draw (0.4,0)
node[above=-1.55pt] {$\down$}; 		\draw (0.8, 0) node[above=-1.55pt] {$\up$}; 		\draw (1.2,0) node[above=-1.55pt]
{$\cdots$}; 		\draw (0.8,0.3) arc (0:180:0.2); 		\begin{scope}[yshift=-0.4cm] 		\draw (-0.6,0)
node[above=-1.55pt] {$\ell_i$}; 		\draw (0.4,0) node[above=-1.55pt] {$0$}; 		\end{scope} 		
\begin{scope}[yshift=0.15cm] 		\scriptsize			\draw (0.6,0) node[above=-1.55pt] {$0$}; 		\end{scope}			
\begin{scope}[yshift=-0.8 cm] 		\draw (-0.6,0) node[above=-1.55pt] {$\la$}; 		\draw (0,0) node[above=-1.55pt]
{$\cdots$}; 		\draw (0.4,0) node[above=-1.55pt] {$\down$}; 		\draw (0.8,0) node[above=-1.55pt] {$\cdots$}; 		
\draw (1.2, 0) node[above=-1.55pt] {$\up$}; 		\draw (1.8,0) node[above=-1.55pt] {$\cdots$}; 		\end{scope} 		
\end{tikzpicture}$$ 		
and therefore $p_{\la,\mu}=q^{j-s}$. By Theorem \ref{Th:constrproj} there is then a unique summand occurring
in the $i$th position of the resolution of $M(\la)$, namely the projective module $P(j-i)$, and we have the distinguished morphism $f_k:=f_{k,k+1}$, homogeneous of degree $1$, from $P(k)$ to $P(k+1)$. Set $\d_{n-k}(n)=(-1)^{n+k+1} f_k$.
\begin{Lemma}\label{Th:diffproj1}
The chain complex
$$0 \to P(0)\langle n\rangle \stackrel{d_{0}}\to P(1)\langle n-1\rangle \to \cdots \stackrel{d_{n-1}}\to P(n)\to 0$$ is a (linear) projective resolution of $M(n)$ in $K_N^1-\gmod$.
\end{Lemma}

\begin{Prop}\label{Prop:IdF}
For $j \geq l$ the identity maps $\operatorname{id}: P(s) \to P(s)$ for all $s\leq l$ define a chain map
\begin{eqnarray*}
\Id^{(j)}_{(l)}:&& P_\bullet(j) \to P_\bullet(l)[j-l]\langle j-l\rangle\\
\end{eqnarray*}
which induces a non-trivial element in $\Ext^{j-l}(M(j),M(l))$. For $j>l$, the maps $f_{s,s-1}:P(s) \to P(s-1)$ for all $s\leq l+1$ define a chain map
\begin{eqnarray*}
F^{(j)}_{(l)}:&& P_\bullet(j) \to P_\bullet(l)[j-l-1]\langle j-l-2\rangle
\end{eqnarray*}
which induces a non-trivial element in $\Ext^{j-l-1}(M(j),M(l))$.
\end{Prop}
\begin{proof}
We have to check that the maps are not nullhomotopic which is clear in the clear in the first case. For $F^{(j)}_{(l)}$, a homotopy would be a map $H \in \hom^{j-l-2}(P_\bullet(j),P_\bullet(l)\langle j-l-2
\rangle)$ which cannot exist by Lemma \ref{Le:mapprojres} since $j \nleq l+1^2+(j-l-2)$.
\end{proof}
The dimension formula from Theorem \ref{Th:dim} implies that we constructed a basis of $E_N^1$. By explicitly composing chain maps we obtain the following relations in $\hom(P_\bullet,P_\bullet)$:
\begin{eqnarray*}
\Id^{(j)}_{(l)} \cdot \Id^{(l)}_{(m)}= \Id^{(j)}_{(m)},\; F^{(j)}_{(l)} \cdot F^{(l)}_{(m)}=0,\; \Id^{(j)}_{(l)} \cdot F^{(l)}_{(m)}=F^{(j)}_{(m)},\; F^{(j)}_{(l)} \cdot \Id^{(l)}_{(m)}= F^{(j)}_{(m)}
\end{eqnarray*}
Reformulating the above result in terms of quivers, we obtain:
\begin{Theorem}\label{mainalg1}
The algebra $E_N^1$ is isomorphic to the path algebra of the quiver
\begin{displaymath}\xymatrix{
(N) \ar@[cyan]@/^/[r] \ar@[black]@/_/[r]&\cdots  \ar@[cyan]@/^/[r] \ar@[black]@/_/[r]& (j+1) \ar@[cyan]@/^/[r]
\ar@[black]@/_/[r]& (j)
 \ar@[cyan]@/^/[r] \ar@[black]@/_/[r]&(j-1)     \ar@/_/[r]  \ar@[cyan]@/^/[r]& \cdots \ar@[cyan]@/^/[r] \ar@[black]@/_/[r]&
 (0) }
\end{displaymath}
\normalsize \color{black} with relations
\begin{displaymath} \xymatrix{
\bullet \ar@[cyan]@/^/[r] &\color{black} \bullet \ar@[cyan]@/^/[r]&\color{black} \bullet \;=\;0,\quad \bullet
\ar@[cyan]@/^/[r] &\bullet \ar@[black]@/_/[r] &\color{black} \bullet \;=\;
 \bullet \ar@[black]@/_/[r] &\color{black}\bullet\ar@[cyan]@/^/[r] &\color{black} \bullet}.
     \end{displaymath}
     \color{black}
The vertex $\bullet$ labeled $i$ corresponds to the idempotent $e_\la$ where $\la=\la_0.s_1 \cdot \dots s_i$.
\end{Theorem}
\color{black}
\subsection{The result for $n=2$}
Now consider $(n,m)=(2,N-1)$. The elements in $W^\p$ are precisely the elements $s_2 \cdot \dots s_k \cdot s_1 \cdot \dots
\cdot s_l$ with $0\leq l < k \leq N$. We denote the weight $\lambda=\la_{0}. s_2 \cdot \ldots \cdot s_k \cdot s_1 \cdot
\ldots \cdot s_l $ by $(k|l)$; the associated diagrammatical weight has $\up$'s at the $l$th and $k$th position
(starting to count with position zero).
\begin{Theorem}\label{mainalg2}
The algebra $E_N^2$ is isomorphic to the path algebra of the quiver which looks as
\begin{displaymath}\xymatrix{
&\cdots \ar@[cyan]@/^/[d] \ar@[black]@/_/[d]& {  \cdots} \ar@[cyan]@/^/[d] \ar@[black]@/_/[d]& \cdots \ar@[cyan]@/^/[d]
\ar@[black]@/_/[d] &
 \\
\cdots  \ar@[cyan]@/^/[r] \ar@[black]@/_/[r]& (k+1|l+1)\ar@[cyan]@/^/[d] \ar@[black]@/_/[d]    \ar@[cyan]@/^/[r]
\ar@[black]@/_/[r]& (k|l+1)\ar@[cyan]@/^/[d] \ar@[black]@/_/[d]
 \ar@[cyan]@/^/[r] \ar@[black]@/_/[r]&(k-1|l+1)   \ar@[cyan]@/^/[d] \ar@[black]@/_/[d]  \ar@/_/[r]  \ar@[cyan]@/^/[r]&
 \cdots  \\
\cdots  \ar@[cyan]@/^/[r] \ar@[black]@/_/[r]& (k+1|l)  \ar@[cyan]@/^/[d] \ar@[black]@/_/[d]     \ar@[cyan]@/^/[r]
\ar@[black]@/_/[r]& (k|l)  \ar@[cyan]@/^/[d] \ar@[black]@/_/[d]     \ar@[cyan]@/^/[r] \ar@[black]@/_/[r]&(k-1|l)
\ar@[cyan]@/^/[d] \ar@[black]@/_/[d]     \ar@[cyan]@/^/[r] \ar@[black]@/_/[r]&\cdots  \\ \cdots  \ar@[cyan]@/^/[r]
\ar@[black]@/_/[r]&(k+1|l-1)  \ar@[cyan]@/^/[d] \ar@[black]@/_/[d]     \ar@[cyan]@/^/[r] \ar@[black]@/_/[r]& (k|l-1)
\ar@[cyan]@/^/[d] \ar@[black]@/_/[d]     \ar@[cyan]@/^/[r] \ar@[black]@/_/[r]&(k-1|l-1)  \ar@[cyan]@/^/[d] \ar@[black]@/_/[d]
\ar@[cyan]@/^/[r] \ar@[black]@/_/[r]& \cdots \\ & \cdots  & \cdots  & \cdots   & }\end{displaymath} for $k>l+2$ and in the other cases:
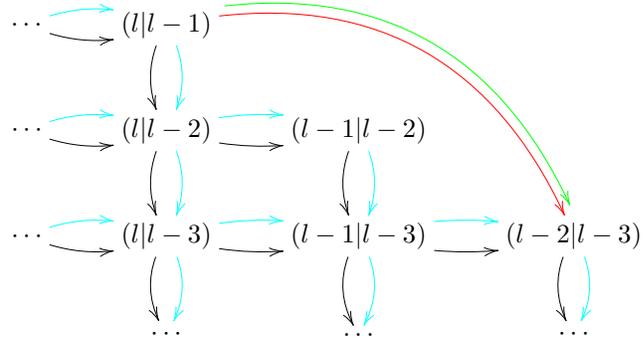
\begin{displaymath}
\xymatrix{ \dots  \ar@[cyan]@/^/[r] \ar@[black]@/_/[r]& (l|l-1) \ar@[cyan]@/^/[d] \ar@[black]@/_/[d] \ar@[red]@/^3pc/ [rrdd]
\ar@[green]@<1ex>@/^3pc/ [rrdd]& &\\ \cdots   \ar@[cyan]@/^/[r] \ar@[black]@/_/[r]&(l|l-2)  \ar@[cyan]@/^/[r]
\ar@[black]@/_/[r] \ar@[cyan]@/^/[d] \ar@[black]@/_/[d]  & (l-1|l-2) \ar@[cyan]@/^/[d] \ar@[black]@/_/[d]  &\\ \cdots
\ar@[cyan]@/^/[r] \ar@[black]@/_/[r]&(l|l-3)   \ar@[cyan]@/^/[r] \ar@[black]@/_/[r] \ar@[cyan]@/^/[d] \ar@[black]@/_/[d]  &
(l-1|l-3)    \ar@[cyan]@/^/[r] \ar@[black]@/_/[r]\ar@[cyan]@/^/[d] \ar@[black]@/_/[d]  & (l-2|l-3)\ar@[cyan]@/^/[d]
\ar@[black]@/_/[d] \\ & \cdots & \dots & \cdots }\end{displaymath} \color{black} with relations as follows (in case that both
sides of the relation exist): 	\begin{multicols}{3}
    \raggedcolumns
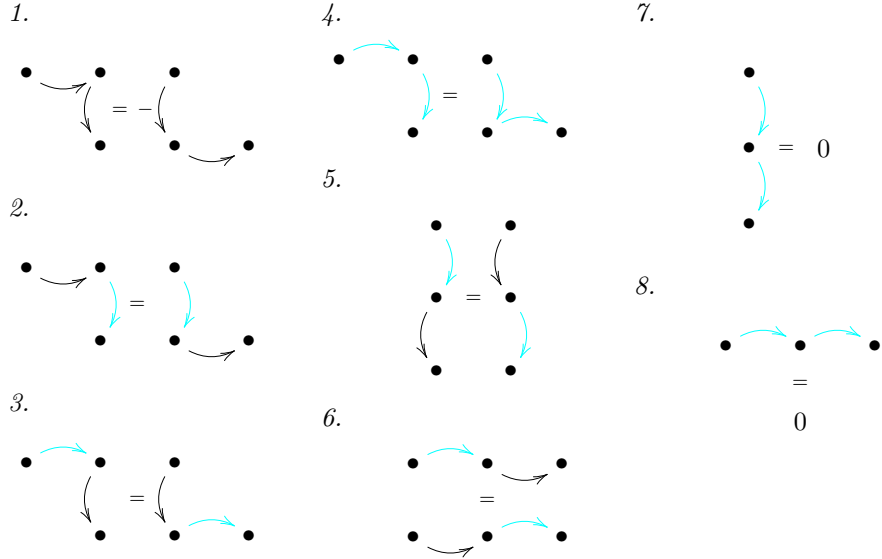
\begin{enumerate}
\item \begin{displaymath}\xymatrix@=0.6cm{\ar @{} [drrr] |{= \ -\ } \bullet \ar@/_/[r] & \bullet \ar@/_/[d]  & \bullet
    \ar@/_/[d] &  \\ &\bullet &\bullet \ar@/_/[r]& \bullet}\end{displaymath}
     \item
    \begin{displaymath}\xymatrix@=0.6cm{\ar @{} [drrr] |{\color{black} =} \color{black} \bullet \ar@[black]@/_/[r]
    &\color{black} \bullet \ar@[cyan]@/^/[d]  & \color{black} \bullet
    \ar@[cyan]@/^/[d] &  \\ &\color{black} \bullet &\color{black} \bullet \ar@[black]@/_/[r]& \bullet}\end{displaymath}
    \item
    \begin{displaymath}\xymatrix@=0.6cm{\ar @{} [drrr] |{\color{black}= } \color{black} \bullet \ar@[cyan]@/^/[r]
    &\color{black} \bullet \ar@[black]@/_/[d]  & \color{black} \bullet
    \ar@[black]@/_/[d] &  \\ &\color{black} \bullet & \color{black}\bullet \ar@[cyan]@/^/[r]& \color{black}
    \bullet}\end{displaymath} \item
    \begin{displaymath}\xymatrix@=0.6cm{\ar @{} [drrr] |{\color{black} = } \color{black} \bullet \ar@[cyan]@/^/[r]
    &\color{black} \bullet \ar@[cyan]@/^/[d]  &
    \color{black} \bullet \ar@[cyan]@/^/[d] &  \\ &\color{black} \bullet &\color{black} \bullet \ar@[cyan]@/^/[r]&
    \color{black} \bullet}\end{displaymath} \item
    \begin{displaymath}\xymatrix@=0.6cm{\ar @{} [ddr] |{\color{black} = } \color{black} \bullet \ar@[cyan]@/^/[d]
    &\color{black}  \bullet \ar@[black]@/_/[d]   \\\color{black}  \bullet
    \ar@[black]@/_/[d] & \color{black} \bullet \ar@[cyan]@/^/[d]\\ \color{black} \bullet &\color{black}
    \bullet}\end{displaymath} \item
    \begin{displaymath}\xymatrix@=0.6cm{\ar @{} [drr] |{\color{black} = } \color{black}\bullet \ar@[cyan]@/^/[r] &
    \color{black} \bullet \ar@[black]@/_/[r] &\color{black} \bullet   \\
    \color{black} \bullet \ar@[black]@/_/[r] &\color{black}  \bullet \ar@[cyan]@/^/[r] &\color{black}
    \bullet}\end{displaymath} \item
    \begin{displaymath}\xymatrix@=0.6cm{\ar @{} [ddr] |{\color{black} \ = \ }\color{black}  \bullet \ar@[cyan]@/^/[d] &
    \\ \color{black} \bullet
    \ar@[cyan]@/^/[d] &\color{black} 0 \\\color{black}  \bullet &}\end{displaymath}
     \item \begin{displaymath}\xymatrix@=0.6cm{\ar @{}[drr] |{\color{black} = }
  \color{black}  \bullet \ar@[cyan]@/^/[r] & \color{black} \bullet \ar@[cyan]@/^/[r] &\color{black} \bullet   \\
  &\color{black} 0  &}\end{displaymath} \color{black}
    \setcounter{enumglobal}{\value{enumi}}
\end{enumerate}
\end{multicols}
\color{black} These are all cases occurring in the middle of the quiver, i.e. in the upper diagram. We also have to look for
those at the corner part. Those can be found in \cite{Klam10}.
\end{Theorem}
\color{black}
\section{The $A_\infty$-structure on $E_m^n$}\label{ch:expAinf}
$A_\infty$-algebras are a generalization of associative algebras, see \cite{Kell2001} for an overview, including historical and topological motivation. A very detailed exposition with most of the proofs is provided in \cite{lefe2003}.

\begin{Def}
An \emph{$A_\infty$-algebra} over a field $k$ is a $\Z$-graded $k$-vector space
$A=\bigoplus_{p \in \Z} A^p$
endowed with a family of graded $k$-linear maps $$m_n: A^{\otimes n} \to A, \ n \geq 1$$
of degree $2-n$ satisfying the following Stasheff identities:
\begin{eqnarray*}
&\sum (-1)^{r+st} m_{r+t+1}(\Id^{\otimes r} \otimes m_s \otimes \Id^{\otimes t}) =0&
\end{eqnarray*}
where for fixed $n$ the sum runs over all decompositions $n=r+s+t$ with $s\geq 1$, and $r,t \geq 0$.
\end{Def}

We use the {\it Koszul sign convention}
$(f \otimes g)(x \otimes y)=(-1)^{|g||x|}f(x) \otimes g(y),$
for tensor products, where $x$, $y$, $f$, $g$ are homogeneous elements of degree $|x|, |y|, |f|, |g|$ respectively.

\begin{Def}
Let $A$ and $B$ be two $A_\infty$-algebras. A \emph{morphism of $A_{\infty}$-algebras} $f:A \to B$ is a family
$f_n: A^{\otimes n} \to B$
of graded $k$-linear maps of degree $1-n$ such that
$$\sum{(-1)^{r+st}f_{r+t+1}(\Id^{\otimes r} \otimes m_s \otimes \Id^{\otimes t})}= \sum{(-1)^w m_q(f_{i_1} \otimes \dots
\otimes f_{i_q})}$$ for all $n \geq 1$. Here, the sum run over all decompositions $n=r+s+t$ and over all decompositions $n=i_1+ \dots +i_q$ with $1 \leq q \leq n$  and all $i_s\geq 1$ respectively. The sign on the right-hand side
is given by $w=\sum_{j=1}^{q-1}(q-j)(i_j-1)$.\\
A morphism $f$ is a \emph{quasi-isomorphism} if $f_1$ is a quasi-isomorphism. It is \emph{strict} if $f_i=0$ for all $i \neq
1$.
\end{Def}
Our goal is to put an $A_\infty$-structure on the $\Ext$-algebras $E_m^n$. The
first step is to introduce an $A_\infty$-structure on the cohomology of an $A_\infty$-algebra (the so-called minimal model) and then realize our Ext-algebra as the cohomology of an
$A_\infty$-algebra, namely the $\hom$-algebra introduced earlier.
\begin{Theorem}[\cite{Kade79}]
Let $A$ be an $A_\infty$-algebra and $H^*(A)$ its cohomology. Then there is an $A_\infty$-structure on $H^*(A)$ such that
$m_1=0$ and $m_2$ is induced by the multiplication on $A$, and there is a quasi-isomorphism of $A_\infty$-algebras $H^*(A) \to
A$ lifting the identity of $H^*(A)$. Moreover, this structure is unique up to isomorphism of $A_\infty$-algebras.
\end{Theorem}
All known (at least to us) proofs inductively construct the model, but the approaches are slightly different. We follow here
Merkulov's more general construction \cite{Merk99} in the special situation of a differential graded algebra:
\begin{Prop}[\cite{Merk99}]\label{Pr:lambdan}
Take $(A,d)$ a differential graded algebra with grading shift $[\quad]$. Let $B \subset A$ be a vector subspace of $A$ and
$\Pi: A \to B$ a projection commuting with $d$. Assume that we are given a homotopy $Q: A \to A[-1]$ such that
\begin{equation} \label{condQ} 1-\Pi=dQ+Qd.
\end{equation}
Define $\la_n: A^{\otimes n} \to A$ for $n \geq 2$ by $\la_2(a_1,a_2):=a_1 \cdot a_2$
and recursively,
\begin{equation}\label{eq:lambda}\begin{split}
&\la_n(a_1,\ldots,a_n)\\ &= - \sum_{\substack{k+l=n\\k,l \geq 1}}{(-1)^{k+(l-1)(|a_1|+\dots+|a_k|)}Q(\la_k(a_1, \ldots, a_k))
\cdot Q(\la_l(a_{k+1}, \ldots , a_n))}
\end{split}.\end{equation}
for $n \geq 3$, setting formally $Q\la_1=-\Id$.
Then the maps $m_1=d$ and $m_n=\Pi(\la_n)$ define an $A_\infty$-structure for a minimal model on $B$.
\end{Prop}
Choosing $Q$ in a clever way simplifies computations, but our result will depend on this choice. We make our choices following \cite{Lu2009}. To define $Q$, we first divide the degree $n$ part $A^n$ of $A$ into three subspaces, for this, denote by $Z^n$
the cocycles of $A$ and by $B^n$ the coboundaries. As we work over a field, we can find subspaces $H^n$ and $L^n$ such that
$Z^n= B^n \oplus H^n$ and
\begin{equation} \label{Identif} A^n=B^n \oplus H^n \oplus L^n.\end{equation}
We identify the $n$th cohomology group $H^n(A)$ via \eqref{Identif} with $H^n$. We want to apply Proposition \ref{Pr:lambdan}
with the choice of a subspace $B=H^*(A)$, the projection $\Pi$ being the projection on the direct summand $H^*$ and the map $Q$ defined as follows:
\begin{enumerate}
\item When restricted to $Z^n$ by equation \eqref{condQ} and the condition that $d|_{Z^n}$ equals to zero, the map $Q$
    has to satisfy the relation
$$1-\Pi=dQ.$$
In particular, $dQ|_{H}$ has to be zero. We choose $Q|_H=0$. \item On $B^n$ the map $\Pi$ is zero, and therefore the map
$Q|_B$ has to satisfy $1=dQ$, i.e. $Q$ has to be a preimage of $d$. We want to choose this preimage as small as possible
i.e. with no non-trivial terms from $Z^n$ (they would anyway be annihilated by $d$). Since $d$ is injective on $L$, we can
choose $Q|_B=(d|_L)^{-1}$. \item  We briefly outline how to determine $Q$ restricted to $L$ (although it won't play any
role in our computations later on). From \eqref{condQ} we get the restriction $$1=Qd+dQ.$$ As $d(a) \in B \text{ for all }
a \in A$ we see that $Qd|_{L}=(d|_L)^{-1}d|_L=1$, so we can define $Q|_{L}=0$.
\end{enumerate}

Now the construction of a minimal model applies to our situation if
we choose $A:=A_m^n:=\hom(P_\bullet, P_\bullet)$, where $P_\bullet$ is the direct sum of all linear projective resolutions of $M(\la)$, $\la\in\La_m^n$ from \ref{Th:constrproj}, and $E=\Ext_m^n=H^*(A)$.

In the following we give an upper bound for the $l$ with $m_l \neq 0$. Already in the case $n=2$ we can show that not all
$m_l$ for $l>2$ vanish and therefore our specific model provides interesting examples of $A_\infty$-algebras with non-trivial
higher multiplications. We start by stating the following Lemma generalizing the fact that the multiplication of two morphisms
can only be non zero if they lie in appropriate $\hom$-spaces.
\begin{Lemma}\label{Le:Laainf}
Let $a_i$, $1 \leq i \leq l$ be homogeneous elements of degree $k_i$ in $E_m^n$ of the form
$$a_i \in \Ext^{k_i}(M(\mu_i),M(\nu_i)) \ 1 \leq i \leq l.$$
Then we have $\la_l(a_1,...,a_l)=0$ unless $\nu_i =\mu_{i+1}$ for all $1 \leq i \leq
l-1$; and if $\la_l(a_1,...,a_l) \neq 0$ we have
$\la_l(a_1,...,a_l) \in \hom^{\Sigma k_i+2-l}(P_\bullet(\mu_1),P_\bullet(\nu_l)).$
\end{Lemma}
\begin{proof}
The proof goes by induction on $l$, using Theorem \ref{Pr:lambdan}, see \cite{Klam10}.
\end{proof}

\begin{Theorem}[General Vanishing Theorem]\label{Th:genvanish}
The $A_\infty$-structure on $E_m^n$ satisfies $m_l=0$ for all $l >n^2+2$.
\end{Theorem}
\begin{proof} We claim that $\la_l=0$ if $l >n^2+2$.
Since $\la_l$ is linear, it is enough to show the assertion on nonzero homogeneous basis elements and therefore by Lemma
\ref{Le:Laainf} we can take $a_i \in \Ext^{k_i}(M(\mu_i),M(\mu_{i+1}))$ for $1 \leq i \leq l.$
By Lemma \ref{Le:Extrest} there are $d_i \geq 0$ such that
$k_i=l(\mu_i)-l(\mu_{i+1})-d_i$ and therefore
$\sum_{i=1}^l{k_i}=l(\mu_1)-l(\mu_{l+1})-\sum_{i=1}^l d_i.$
>From Lemma~\ref{Le:Laainf} we know that $\la_l(a_1,...,a_l) \in \hom^{\Sigma k_i+2-l}(P_\bullet(\mu_1),P_\bullet(\nu_l)).$
Assume $\la_l \neq 0$, so, by Lemma \ref{Le:mapprojres} about the morphisms between our chosen projective resolutions, we know
that
$l(\mu_1) \leq l(\mu_{l+1})+n^2 +\sum k_i+2-l,$
thus
\begin{eqnarray*}
l(\mu_1) &\leq& l(\mu_{l+1})+l(\mu_1)-l(\mu_{l+1})-\sum_{i=1}^l {d_i}+2-l +n^2,
\end{eqnarray*}
which is equivalent to $\sum_{i=1}^l {d_i} \leq n^2+2 -l$.
Since $\sum_{i=1}^l {d_i} \geq 0$, we get  $0\leq n^2+2-l$, equivalently $l\leq n^2+2$; providing the
asserted upper bound.
\end{proof}
\subsection{Explicit computations for $E_N^{1}$ and $E_{N-1}^2$}
In the previous section we established general vanishing results for the higher multiplications; in this section we describe
explicit models for our small examples $n=1$ and $n=2$.
The first result in this situation is the following:
\begin{Theorem}[1st vanishing Theorem] \label{Th:1stvanish}
The algebra $E_1^N$ is formal, i.e. there is a minimal model such that $m_n =0$ for all $n \geq 3$.
\end{Theorem}
\begin{proof}
Recall that all multiplication rules in the algebra $E_N^1$ are already determined in $A_N^1=\hom(P_\bullet, P_\bullet)$. Therefore, for all
elements $a_1, a_2 \in \Ext(\oplus M(\la), \oplus M(\la))=H^*(\hom(P_\bullet, P_\bullet))$ identified with the subspace $H^*$
via the decomposition from \eqref{Identif}, the product $a_1 \cdot a_2$ also lies in the subspace $H^*$ and has no boundary
component in $B^*$. Since we have chosen $Q|_H=0$, we obtain
$Q(a_1 \cdot a_2)=0.$
Using the construction of the higher multiplications in Proposition \ref{Pr:lambdan} one gets
$m_n = 0$ for all $n \geq 3$.
\end{proof}
\subsubsection{The case $E_{N-1}^2$}\label{sec:Exp2}
The case of $n=2$ turns out to be more interesting than the case $n=1$ studied before, since we have non-vanishing higher multiplications. In contrast to the previous example this
phenomenon is possible, since some multiplications in $A_{N-2}^1=\hom(P_\bullet, P_\bullet)$ are only homotopic to their product in the
$\Ext$-algebra. This yields the following theorem:
\begin{Theorem}\label{Th:la3}
In the minimal model above, there are non-vanishing $m_3$.
\end{Theorem}
A complete list of all higher multiplications $m_3$ is given in \cite{Klam10}.
\subsubsection{Vanishing of higher multiplications}
Detailed knowledge about the structure of projective resolutions provides a stronger
vanishing result than in the general case (see \cite{Klam10}):
\begin{Theorem}[2nd Vanishing Theorem]
\label{Th:2ndvanish} The $A_\infty$-structure on $E_{N-2}^2$ given by the construction above satisfies
$$m_n=0 \ \forall n \geq 4.$$
\end{Theorem}
\subsection{Ideas how to prove non-formality}
In the previous section we proved that there is a minimal model with non-vanishing higher multiplications but this does not
answer the question whether the algebra is formal. To show that the algebra is not formal, we have to prove that no model
exists such that $m_n=0$ for all $n \geq 3$. As a tool one could use Hochschild cohomology. Given a dg-Algebra $A$ one can
compute its Hochschild cohomology by using the $A_\infty$-structure on a minimal model of $A$ (cf. \cite[Lemma
B.4.1]{lefe2003} and \cite{kade1988}). Assume that we have found a minimal model on $H^*(A)$ with $m_n=0$ for $3 \leq n \leq
p-1$. Then the multiplication $m_p$ defines a cocycle for the Hochschild cohomology of $A$ by the construction in \cite[Lemma
B.4.1]{lefe2003}. If we can prove that this class is not trivial, we are done and have shown that the algebra is not formal.
If we cannot, we have to modify our model such that $m_p=0$ and then analyze if $m_{p+1}$ vanishes. A detailed discussion of
this topic would go beyond the scope of this article. Therefore we only state the following conjecture:
\begin{Conjecture}
In general, the algebra $E_m^n$ is not formal.
\end{Conjecture}
\FloatBarrier


\bibliographystyle{elsarticle-harv}







\end{document}